\documentclass[11pt]{article}
\usepackage{latexsym,amsmath,amssymb}

\newif\ifdviwin

\usepackage{enumerate}
\usepackage[latin1]{inputenc}
\usepackage[english]{babel}
\usepackage{indentfirst}
\usepackage[mathscr]{eucal}
\usepackage{amssymb,amsmath,amsfonts}
\usepackage{graphicx}

\usepackage{graphics}
\usepackage{times} 
\usepackage{amsthm} 

\newif\ifdviwin

\dviwintrue

\def\esf{\mathbb{S}}
\def\R{\mathbb{R}}

\def\N{\mathbb{N}}
\def\D{\mathbb{D}}
\def\ee{\textsc{e}}
\def\J{\mathcal{J}}
\def\I{\mathcal{I}}

\def\cM{\mathcal{M}}

\newcommand{\longui}{\operatorname{length}}
\newcommand{\dist}{\operatorname{dist}}
\newcommand{\intc}{\operatorname{Int}}

\def\a{{\alpha}}

\def\g{{\gamma}}

\def\l{{\lambda}}
\def\de{{\delta}}

\def\ve{{\varepsilon}}
\def\vp{{\varphi}}
\def\s{{\sigma}}
\def\ep{{\epsilon}}

\newtheorem{lemma}{Lemma}
\newtheorem{remark}{Remark}
\newtheorem{theorem}{Theorem}

\newtheorem{corollary}{Corollary}
\newtheorem{definition}{Definition}

\newtheorem*{theoremintro}{Theorem}

\begin{document}
\mbox{}\vspace{0.4cm}\mbox{}

\begin{center}
\rule{14cm}{1.5pt}\vspace{0.5cm}

{\Large \bf Compact complete proper minimal immersions} \\
[0.3cm]{\Large \bf
in strictly convex bounded regular domains of $\R^3$}\\
\vspace{0.5cm} {\large Antonio Alarc\'{o}n\footnote[0]{Research
partially supported by Spanish MEC-FEDER Grant MTM2007-61775 and
Regional
 J. Andaluc\'{i}a Grant P06-FQM-01642.} }

\vspace{0.3cm} \rule{14cm}{1.5pt}
\end{center}

\vspace{0.5cm}

\begin{quote}
{\small
\noindent {\bf Abstract}\quad Consider a strictly convex
bounded regular domain $C$ of $\R^3$. For any arbitrary finite
topological type we find a compact Riemann surface $\cM$, an open
domain $M\subset \cM$ with the fixed topological type, and a
conformal complete proper minimal immersion $X:M\to C$ which can
be extended to a continuous map $X:\overline{M}\to \overline{C}$.
\\


}\end{quote}


\section{Introduction}\label{sec: intro}

The global theory of complete minimal surfaces in $\R^3$ has been
developed for almost two and one-half centuries. One of the
central questions in this theory has been the Calabi-Yau problem,
which dates back to the 1960s. Calabi \cite{C} asked whether or
not it is possible for a complete minimal surface in $\R^3$ to be
bounded. The most important result in this line is due to
Nadirashvili \cite{N1}, who constructed a complete minimal disk in
a ball. After Nadirashvili's answer, Yau \cite{Y} stated new
questions related to the embeddedness and properness of surfaces
of this type.

Concerning the embedded question, Colding and Minicozzi \cite{CM}
proved that a complete embedded minimal surface with finite
topology in $\R^3$ must be properly embedded. In particular, it
must be unbounded. This result was generalized in two different
directions. On the one hand, Meeks, P\'{e}rez and Ros \cite{MPR}
proved that any complete embedded minimal surface in $\R^3$ with
finite genus and countably many ends must be proper in the space.
On the other hand, Meeks and Rosenberg \cite{MR} showed that a
complete embedded minimal surface with positive injectivity radius
is proper in $\R^3$.

Regarding the properness of the examples, Mart\'{i}n and Morales
\cite{MM1} introduced an additional ingredient into Nadirashvili's
technique to prove that every convex domain and every regular
bounded domain admits a complete properly immersed minimal disk.
Moreover, they showed that the limit set of such surfaces can be
chosen close to a prescribed smooth Jordan curve on the boundary
of the domain \cite{MM2}. These arguments were generalized by
Ferrer, Mart\'{i}n and the author \cite{AFM} to prove that any
open Riemann surface of finite topology can be properly and
minimally immersed in any convex domain of $\R^3$ or any bounded
and smooth one. The same conclusion for any open Riemann surface
was proved by Ferrer, Mart\'{i}n and Meeks \cite{FMM}. In contrast
to these existence results, Mart\'{i}n, Meeks and Nadirashvili
\cite{MMN} proved the existence of bounded open regions of $\R^3$
which do not contain a complete properly immersed minimal surface
with finite topology.

The study of the Calabi-Yau problem gave rise to new lines of work
and techniques. Among other things, these new ideas established a
surprising relationship between the theory of complete minimal
surfaces in $\R^3$ and the Plateau problem. This problem consists
of finding a minimal surface spanning a given family of closed
curves in $\R^3,$ and it was solved independently by Douglas
\cite{D} and Rad\'{o} \cite{R}, for any Jordan curve. The link
between complete minimal surfaces and the Plateau problem is the
existence of compact complete minimal immersions in $\R^3$,
according to the following definition (see \cite{AN,A}).
\begin{definition}\label{def: compact} By a compact minimal
immersion we mean a minimal immersion $X:M\to\R^3,$ where $M$ is
an open region of a compact Riemann surface $\mathcal{M},$ and
such that $X$ can be extended to a continuous map
$X:\overline{M}\to\R^3.$
\end{definition}
Mart\'{i}n and Nadirashvili \cite{MN} constructed compact complete
conformal minimal immersions $X:\D\to\R^3$ such that $X_{|\partial
\D}$ is an embedding and $X(\esf^1)$ is a Jordan curve with
Hausdorff dimension $1.$ Furthermore, they showed that the set of
Jordan curves $X(\esf^1)$ constructed by the above procedure is
dense in the space of Jordan curves of $\R^3$ with the Hausdorff
distance. After this, the author constructed compact complete
minimal immersions of Riemann surfaces of arbitrary finite
topology \cite{A}. As in the simply connected case, the set of
closed curves given by the limit sets of these immersions is dense
in the space of finite families of closed curves in $\R^3$ which
admit a solution to the Plateau problem. In spite of these density
theorems, there are some requirements for the limit set of a
compact complete minimal immersion. Nadirashvili and the author
\cite{AN} proved that there is no compact complete proper minimal
immersion of the disk into a polyhedron of $\R^3$ (see \cite{N2}
for the case of a cube). In fact, given $D\subset\R^3$ a regular
domain and $X:\D\to D$ a compact complete proper minimal
immersion, then the second fundamental form of the surface
$\partial D$ at any point of the limit set of $X$ must be
nonnegatively definite \cite{AN}.

The aim of the present paper is to join the techniques used in the
construction of complete proper minimal surfaces with finite
topology in convex domains of $\R^3$, and those used to construct
compact complete minimal immersions, in order to prove the
following result.

\begin{theoremintro}\label{th:intro}
For any $C$ strictly convex bounded regular domain of $\R^3$,
there exist compact complete proper minimal immersions $X:M\to C$
of arbitrary finite topological type.

Moreover, for any finite family $\Sigma$ of closed curves in
$\partial C$ which admits a solution to the Plateau problem, and
for any $\xi>0$, there exists a minimal immersion $X:M\to C$ in
the above conditions and such that $\de^H (\Sigma,X(\partial
M))<\xi$, where $\de^H$ means the Hausdorff distance.
\end{theoremintro}

Since the existence of solution of the Plateau problem for any
Jordan curve, it is deduced that the above result has specially
interesting consequences for disks. Given $C$ a strictly convex
bounded regular domain of $\R^3$, then any Jordan curve in
$\partial C$ can be approximated in terms of the Hausdorff
distance by the limit set of a compact complete proper minimal
immersion of the disk into $C$ (Subsection \ref{subsec: disk}).

Finally, we would like to point out that our result is sharp in
the following sense. If we remove the hypothesis of $C$ being
strictly convex, then the theorem fails \cite{MMN,AN}.


\section{Preliminaries and background}\label{sec: preliminaries}

Here we briefly summarize the notation and results that we use in
the paper.

\subsection{Riemann surfaces}\label{sucsec: riemann}

Throughout the paper we work on a compact Riemann surface endowed
with a Riemannian metric. We consider that the following data are
fixed.

\begin{definition}\label{def: M-ds}
Let $M'$ be a compact Riemann surface of genus $\s\in\N \cup
\{0\},$ and $ds^2$ a Riemannian metric in $M'.$
\end{definition}

Consider a subset $W \subset M',$ and a Riemannian metric
$d\tau^2$ in $W.$ Given a curve $\a$ in $W,$ by
$\longui_{d\tau}(\a)$ we mean the length of $\a$ with respect to
the metric $d\tau^2.$ Moreover, we define:
\begin{enumerate}[1]
\item[$\bullet$] $\dist_{(W,d\tau)}(p,q)=\inf \{\longui_{d\tau}(\alpha)
\: | \: \alpha:[0,1]\rightarrow W, \; \alpha(0)=p,\alpha(1)=q \}$, for any $p,q\in W.$
\item[$\bullet$] $\dist_{(W,d\tau)}(T_1,T_2)=\inf \{\dist_{(W,d\tau)}(p,q)
\;|\;p \in T_1, \;q \in T_2 \}$, for any $T_1, T_2 \subset  W.$
\end{enumerate}
We usually work with a domain $W$ in $M'$ and a conformal minimal
immersion $Y:\overline{W}\to\R^3.$ Then, by $ds_{Y}^2$ we mean the
Riemannian metric induced by $Y$ in $\overline{W}.$ Moreover, we
write $\dist_{(\overline{W},Y)}(T_1,T_2)$ instead of
$\dist_{(\overline{W},ds_Y)}(T_1,T_2),$ for any sets $T_1$ and
$T_2$ in $\overline{W}.$

Given $\textsc{e}\in\N$, consider
$\D_1,\ldots,\D_\textsc{e}\subset M'$ open disks such that
$\{\g_i:=\partial \D_i\}_{i=1}^{\textsc{e}}$ are analytic Jordan
curves and $\overline{\D}_i\cap \overline{\D}_j=\emptyset$ for all
$i\neq j$.

\begin{definition}\label{def: multicycle}
Each curve $\g_i$ is called a cycle on $M'$ and the family
$\mathcal{ J}=\{\g_1,\ldots,\g_\textsc{e}\}$ is called a
multicycle on $M'$. We denote by $\intc(\g_i)$ the disk $\D_i$,
for $i=1,\ldots, \textsc{e}.$ We also define $M(\mathcal{
J})=M'\setminus(\cup_{i=1}^\textsc{e} \overline{\intc(\g_i)})$.
\end{definition}
In this setting $M(\mathcal{J})$ is a hyperbolic Riemann surface
with genus $\sigma$ and $\ee$ ends. Most of our immersions will be
defined over surfaces constructed by this way.

Given $\mathcal{ J}=\{\g_1,\ldots,\g_\textsc{e}\}$ and $\mathcal{
J}'=\{\g_1',\ldots,\g_\textsc{e}'\}$ two multicycles on $M'$ we
write $\mathcal{ J}'< \mathcal{ J}$ if $\overline{\intc(\g_i)}
\subset \intc (\g_i')$ for $i=1,\ldots, \textsc{e}.$ Notice that
$\mathcal{ J}'< \mathcal{ J}$ implies $\overline{M(\mathcal{ J}')}
\subset M(\mathcal{ J})$.

Let $\mathcal{ J}=\{\g_1,\ldots,\g_\textsc{e}\}$ be a multicycle
on $M'.$ If $\ep>0$ is small enough, we can consider the
multicycle
$\mathcal{J}^\ep=\{\g_1^\ep,\ldots,\g_\textsc{e}^\ep\},$ where by
$\g_i^\ep$ we mean the cycle satisfying $\overline{\intc
(\g_i)}\subset\intc(\g_i^\ep)$ and $\dist_{(M',ds)}(q,\g_i)=\ep$
for all $q\in \g_i^\ep.$ Notice that $\J^\ep<\J.$


\subsection{Convex domains and Hausdorff distance}\label{subsec:
convex}

Given $E$ a bounded regular convex domain of $\R^3$, and
$p\in\partial E$, we let $\kappa_2(p)\geq \kappa_1(p)\geq 0$
denote the principal curvatures of $\partial E$ at $p$ associated
to the inward pointing unit normal. Moreover, we write
\[
\kappa_1(\partial E):= \min \{\kappa_1(p)\;|\;p\in\partial E\}\geq
0 .
\]
If $E$ is in addition strictly convex, then $\kappa_1(\partial
E)>0$. Recall that $E$ is strictly convex if, and only if, the
principal curvatures of $\partial E$ associated to the inward
pointing unit normal are positive everywhere.

If we consider $\mathcal{N}: \partial E \rightarrow \esf^2$ the
outward pointing unit normal or Gauss map of $\partial E$, then
there exists a constant $a>0$ (depending on $E$) such that
$\partial E_t=\{p+ t\cdot \mathcal{N}(p) \; | \; p \in \partial
E\}$ is a regular (convex) surface $\forall t \in [-a, +\infty[$.
We label $E_t$ as the convex domain bounded by $\partial E_t$.

The set $\mathcal{C}^n$ of convex bodies of $\R^n$, i.e. convex
compact sets of $\R^n$ with {non\-empty} interior, can be made into a
metric space in several geometrically reasonable ways. The
Hausdorff metric is particularly convenient and applicable. The
natural domain for this metric is the set $\mathcal{K}^n$ of the
nonempty compact subsets of $\R^n$. Given $C,D\in\mathcal{K}^n$,
the Hausdorff distance between $C$ and $D$ is defined by
\[
\de^H(C,D)=\max\left\{ \sup_{x\in C} \inf_{y\in D} \|x-y\|\;,\;
\sup_{y\in D} \inf_{x\in C} \|x-y\|\right\}.
\]


\subsection{Preliminary lemma}\label{subsec: lemma}

Next lemma was proved by Ferrer, Mart\'{i}n and the author
\cite[Lemma 5]{AFM}. However, its usefulness in the construction
of compact complete minimal immersions has been entirely developed
and exploited in this paper.

\begin{lemma}\label{lem: main-lemma}
Let $\J$ be a multicycle on $M'$, $X:\overline{M(\J)}\to\R^3$ a
conformal minimal immersion, and $p_0\in M(\J)$ with $X(p_0)=0$.
Consider $E$ a strictly convex bounded regular domain, and $E'$ a
convex bounded regular domain, with $0\in
E\subset\overline{E}\subset E'$. Let $a$ and $\ep$ be positive
constants satisfying that $p_0\in M(\J^\ep)$ and
\begin{equation}\label{eq:michicha}
X(\overline{M(\J)}\setminus M(\J^\ep))\subset E\setminus \overline{E_{-a}}.
\end{equation}

Then, for any $b>0$ there exist a multicycle $\widehat{\J}$ and a
conformal minimal immersion
$\widehat{X}:\overline{M(\widehat{\J})}\to \R^3$ with the
following properties:
\begin{enumerate}[\rm ({L}1)]
\item $\widehat{X}(p_0)=0$.

\item $\J^\ep < \widehat{\J} < \J$.

\item $1/\ep < \dist_{(\overline{M(\widehat{\J})},\widehat{X})}(p,\J^\ep)$, $\forall p\in \widehat{\J}$.

\item $\widehat{X}(\widehat{\J})\subset E'\setminus \overline{E'_{-b}}$.

\item $\widehat{X}(\overline{M(\widehat{\J})}\setminus M(\J^\ep))\subset \R^3\setminus E_{-2b-a}$.

\item $\|\widehat{X}-X\|<\ep$ in $\overline{M(\J^\ep)}$.

\item $\|\widehat{X}-X\|< {\tt m}(a,b,\ep,E,E')$ in $\overline{M(\widehat{\J})}$, where
\[
{\tt m}(a,b,\ep,E,E'):=\ep+\sqrt{\frac{2(\de^H(E,E')+a+2b)}{\kappa_1(\partial E)}+(\de^H(E,E')+a)^2}.
\]

\end{enumerate}
\end{lemma}

The above lemma essentially asserts that a minimal surface of
finite topology with boundary can be perturbed outside a compact
set (see (L6)) in such a way that the intrinsic diameter of the
surface grows (see (L3)), and the boundary of the resulting
surface achieves the boundary of a prescribed convex domain (see
(L4)). Moreover, the deformation keeps the perturbed part of the
surface outside another prescribed convex domain (see
\eqref{eq:michicha} and (L5)). Finally, the domain of definition
of the perturbed immersion is contained in the one of the original
immersion (see (L2)), and there is an upper bound for the
difference between both immersions in the whole domain of
definition of the deformed one (see (L7)). Let us point out that
the assumption of $E$ being strictly convex is essential in this
lemma, otherwise statement (L7) has no sense.

The proof of Lemma \ref{lem: main-lemma} is based on the use of a
Runge's type theorem for compact Riemann surfaces and the
L\'{o}pez-Ros transformation for minimal surfaces.
For a detailed proof of this lemma we refer the reader to
\cite{AFM}. We have stated it here just to make this paper
self-contained.


\section{The main theorem}\label{sec: theorem}

In this section we prove the main result of this paper and derive
some corollaries. From now on $C$ represents a strictly convex
bounded regular domain of $\R^3$.

The theorem stated in the introduction trivially follows from the
following one.

\begin{theorem}\label{th: main}
Let $C$ be a strictly convex bounded regular domain of $\R^3.$
Consider $\J$ a multicycle on the Riemann
surface $M'$ and $\phi:\overline{M(\J)}\to\overline{C}$ a
conformal minimal immersion satisfying $\phi(\J)\subset \partial
C$.

Then, for any $\mu>0,$ there exist a domain $M_\mu$ and a complete
proper conformal minimal immersion $\phi_\mu:M_\mu\to C$ such
that:
\begin{enumerate}[\rm (i)]
\item $\overline{M(\J^\mu)}\subset M_\mu\subset
\overline{M_\mu}\subset M(\J),$ and $M_\mu$ has the topological type of $M(\J).$

\item $\phi_\mu$ admits a continuous extension $\Phi_\mu:\overline{M_\mu}\to\overline{C}$
and $\Phi_\mu(\partial M_\mu)\subset\partial C.$

\item $\|\phi-\Phi_\mu\|<\mu$ in $\overline{M_\mu}.$

\item $\delta^H(\phi(\overline{M(\J)}),\Phi_\mu(\overline{M_\mu}))<\mu.$

\item $\delta^H(\phi(\J),\Phi_\mu(\partial M_\mu))<\mu.$
\end{enumerate}
\end{theorem}

The proof of the above theorem consists roughly of the following.
First we look for an exhaustion sequence $\{E^n\}_{n\in\N}$ of
strictly convex bounded regular domains covering $C$. Then we use
Lemma \ref{lem: main-lemma} in a recursive way in order to
construct a sequence of minimal immersions $\{X_n\}_{n\in\N}$,
starting at $X_1=\phi$. To construct the immersion $X_{n+1}$ we
apply the lemma to the data $X=X_n$, $E=E^n$, $E'=E^{n+1}$ and
constants $a=b_n$, $b=b_{n+1}$ and $\ep=\ep_{n+1}$. These
constants and the convex domains $\{E^n\}_{n\in\N}$ are suitably
chosen so that the sequence $\{X_n\}_{n\in\N}$ has a limit
immersion $\phi_\mu$ which satisfies the conclusion of Theorem
\ref{th: main}. The completeness of $\phi_\mu$ is derived from
property (L3) of Lemma \ref{lem: main-lemma}. The properness in
$C$ follows from (L4) and (L5). Finally, to guarantee that the
immersion $\phi_\mu$ is compact we use property (L7).

\subsection{Proof of Theorem \ref{th: main}}

Assume $\J=\{\g_1,\ldots,\g_\ee\}$. First of all, we define a
positive constant $\ve<\mu/2$. In order to do it, consider
$\mathcal{T}(\g_i)$ a tubular neighborhood of $\g_i$ in
$\overline{M(\J)}$, and denote by ${\tt
P}_i:\mathcal{T}(\g_i)\to\g_i$ the natural projection,
$i=1,\ldots,\ee$. Choose $\ve>0$ small enough so that
$\overline{M(\J)}\setminus M(\J^\ve) \subset\cup_{i=1}^\ee
\mathcal{T}(\g_i),$ and
\begin{equation}\label{equ: ve<mu 2}
\|\phi(p)-\phi({\tt P}_i(p))\|<\frac{\mu}2,
\end{equation}
for any $p\in ( \overline{M(\J)}\setminus M(\J^\ve) )\cap
\mathcal{T}(\g_i)$, $i=1,\ldots,\ee$. This choice is possible
since the uniform continuity of $\phi$. The definition of $\ve$ is
nothing but a trick to obtain statements (iv) and (v) from
statement (iii).

Now, let us describe how to define the family $\{E^n\}_{n\in\N}$
of convex sets. Consider $t_0>0$ small enough so that, for any
$t\in]0,t_0[$,
\begin{enumerate}[\rm 1]
\item[$\bullet$] $C_t$ is a well defined strictly convex bounded regular domain.
\item[$\bullet$] $\Gamma_t:=\phi^{-1}((\partial C_t)\cap \phi(M(\J)))$ is a multicycle on $M'$.
\end{enumerate}
Let $c_1$ be a positive constant (which will be specified later)
small enough so that
\[
c_1^2\cdot \sum_{k\geq 1}\frac1{k^4}<\min\{t_0,\ve\},
\]
and define, for any natural $n$,
\begin{equation}\label{equ: t}
t_n:=c_1^2\cdot\sum_{k\geq n} \frac{1}{k^4}.
\end{equation}
Then, $\forall n\in\N$, we consider the strictly convex bounded
regular domain
\begin{equation}\label{equ: E}
E^n:=C_{-t_n}.
\end{equation}
Notice that $E^n\subset E^{n+1}$, $\forall n\in\N$, and
$\cup_{n\in\N}E^n=C.$ Furthermore, from \eqref{equ: t} and
\eqref{equ: E}, the Hausdorff distance between $E^n$ and $E^{n+1}$
is known. In fact
\begin{equation}\label{eq:anita}
\de^H(E^{n-1},E^n)=\frac{c_1^2}{n^4},\quad \forall n\in\N.
\end{equation}

Finally, consider a decreasing sequence of positives
$\{b_n\}_{n\in\N}$ satisfying
\begin{equation}\label{equ: tesis-7.7}
b_1<2(t_0-t_1),\quad \text{and}\quad b_n<\frac{c_1^2}{n^4}.
\end{equation}
These numbers will take the role of the constants $a$ and $b$ of
Lemma \ref{lem: main-lemma} in the recursive process.

The next step consists of using Lemma \ref{lem: main-lemma} to
construct, for any $n\in\N$, a family
$\chi_n=\{\J_n,X_n,\ep_n,\xi_n\}$,
where
\begin{enumerate}[\rm 1]
\item[$\bullet$] $\J_n$ is a multicycle on $M'$.

\item[$\bullet$] $X_n:\overline{M(\J_n)}\to C$ is a conformal minimal immersion.

\item[$\bullet$] $\{\ep_n\}_{n\in\N}$ and $\{\xi_n\}_{n\in\N}$
are decreasing sequences of positive real numbers with
\begin{equation}\label{equ: tesis-7.8}
\xi_n<\ep_n<\frac{c_1}{n^2}.
\end{equation}
\end{enumerate}

Moreover, the sequence of families $\{\chi_n\}_{n\in\N}$ must
satisfy the following list of properties:
\begin{enumerate}[\rm (A$_{n}$)]
\item $\J^\ve<\J_{n-1}^{\xi_{n-1}}< \J_{n-1}^{\ep_n}< \J_n^{\xi_n}<\J_n<\J_{n-1}$.

\item $1/\ep_n< \dist_{(\overline{M(\J_n^{\xi_n})},X_n)}(\J_{n-1}^{\xi_{n-1}}, \J_n^{\xi_n})$.

\item $\|X_n-X_{n-1}\|<\ep_n$ in $\overline{M(\J_{n-1}^{\ep_n})}$.

\item $ds_{X_n}\geq \alpha_n\cdot ds_{X_{n-1}}$ in $\overline{M(\J_{n-1}^{\xi_{n-1}})}$,
where the sequence $\{\a_k\}_{k\in\N}$ is given by
\[
\a_1:=\frac12\, e^{1/2},\quad \a_k:=e^{-1/2^k}\text{ for }k>1.
\]
Notice that $0<\a_k<1$ and $\{\prod_{m=1}^k \a_m\}_{k\in\N}$
converges to $1/2$.

\item $X_n(p)\in E^n\setminus \overline{(E^n)_{-b_n}}$, for any $p\in \J_n$.

\item $X_n(p)\in\R^3\setminus (E^{n-1})_{-b_{n-1}-2b_n}$, for any $p\in \overline{M(\J_n)}\setminus M(\J_{n-1}^{\ep_n})$.

\item $\|X_n-X_{n-1}\|<{\tt m}(b_{n-1},b_n,\ep_n,E^{n-1},E^n)$ in $\overline{M(\J_n)}$, where ${\tt m}$ is the map defined in Lemma \ref{lem: main-lemma}.

\end{enumerate}

The sequence $\{\chi_n\}_{n\in\N}$ is constructed in a recursive
way. To define $\chi_1$, we choose $X_1=\phi$ and
$\J_1=\Gamma_{-t_1-b_1/2}$. The first inequality of \eqref{equ:
tesis-7.7} guarantees that $\J_1$ is well defined. From this
choice we conclude that $X_1(\J_1)\subset \partial
C_{-t_1-b_1/2}\subset E^1\setminus\overline{(E^1)_{-b_1}}$, and so
property (E$_1$) holds. Then, we take $\ep_1$ and $\xi_1$
satisfying \eqref{equ: tesis-7.8} and being $\xi_1$ small enough
so that $\J^\ve<\J_1^{\xi_1}$. The remainder properties of the
family $\chi_1$ do not make sense. The definition of $\chi_1$ is
done.

Now, assume that we have constructed the families
$\chi_1,\ldots,\chi_n$ satisfying the desired properties. Let us
show how to construct $\chi_{n+1}$. First of all, notice that
property (E$_n$) guarantees the existence of a positive constant
$\l$ such that $X_n(\overline{M(\J_n)}\setminus
M(\J_n^{\l}))\subset E^n\setminus \overline{(E^n)_{-b_n}}$. Then,
Lemma \ref{lem: main-lemma} can be applied to the data
\[
\J=\J_n,\quad X=X_n,\quad E=E^n,\quad E'=E^{n+1}, \quad
a=b_n,\quad\ep, \quad b=b_{n+1},
\]
for any $0<\ep<\l$. Now, consider a sequence of positives
$\{\widehat{\ep}_m\}_{m\in\N}$ decreasing to zero and such that
\begin{equation}\label{equ: ep-gorro}
\widehat{\ep}_m<\min \{\l,\xi_n,c_1/(n+1)^2\},\quad \text{for any
$m\in\N$}.
\end{equation}
Consider $\I_m$ and $Y_m:\overline{M(\I_m)}\to \R^3$ the
multicycle and the conformal minimal immersion given by Lemma
\ref{lem: main-lemma} for the above data and
$\ep=\widehat{\ep}_m$. Statement (L2) in Lemma \ref{lem:
main-lemma} and \eqref{equ: ep-gorro} imply that
\begin{equation}\label{equ: gafa-7.2}
\J_n^{\xi_n}<\J_n^{\widehat{\ep}_m}<\I_m,\quad \text{for any
$m\in\N$}.
\end{equation}
Taking \eqref{equ: gafa-7.2} into account, (L6) guarantees that
the sequence $\{Y_m\}_{m\in\N}$ converges to $X_n$ uniformly in
$\overline{M(\J_n^{\xi_n})}$. In particular, the sequence of
metrics $\{ds_{Y_m}\}_{m\in\N}$ converges to $ds_{X_n}$ uniformly
in $\overline{M(\J_n^{\xi_n})}$. Therefore, there exists
$m_0\in\N$ large enough so that
\begin{equation}\label{equ: gafa-7.3}
ds_{Y_{m_0}}\geq \a_{n+1}\cdot ds_{X_n}\quad \text{in
$\overline{M(\J_n^{\xi_n})}$}.
\end{equation}
Define $\J_{n+1}:=\I_{m_0}$, $X_{n+1}:=Y_{m_0}$, and
$\ep_{n+1}:=\widehat{\ep}_{m_0}$. From \eqref{equ: gafa-7.2} and
statement (L3) in Lemma \ref{lem: main-lemma} we deduce that
\[
1/\ep_{n+1}<\dist_{(\overline{M(\J_{n+1})},X_{n+1})}(\J_{n+1},\J_n^{\xi_n}).
\]
Then, the above equation and \eqref{equ: gafa-7.2} imply the
existence of a positive $\xi_{n+1}$ small enough so that
\eqref{equ: tesis-7.8}, (A$_{n+1}$) and (B$_{n+1}$) hold.
Properties (C$_{n+1}$), (D$_{n+1}$), (E$_{n+1}$), (F$_{n+1}$) and
(G$_{n+1}$) follow from (L6), \eqref{equ: gafa-7.3}, (L4), (L5)
and (L7), respectively. The definition of $\chi_{n+1}$ is done.

In this way the sequence $\{\chi_n\}_{n\in\N}$ satisfying the
desired properties has been constructed.

The next step consists of defining the domain of definition of the
immersion which satisfies the conclusion of the theorem. Consider
\[
M_\mu:=\bigcup_{n\in\N}M(\J_n^{\ep_{n+1}})=\bigcup_{n\in\N}M(\J_n^{\xi_n}).
\]
Since (A$_n$), $n\in\N$, the set $M_\mu$ is an expansive union of
domains with the same topological type as $M(\J)$. Therefore,
elementary topological arguments give that $M_\mu$ is a domain
with the same topological type as $M(\J)$. Furthermore, (A$_n$),
$n\in\N$, also imply that
\begin{equation}\label{equ: cierre}
\overline{M_\mu}=\bigcap_{n\in\N} \overline{M(\J_n)}.
\end{equation}

Now we specify the constant $c_1$. We take $c_1$ small enough so
that
\begin{equation}\label{equ: suma m}
\sum_{n=2}^\infty {\tt
m}(b_{n-1},b_{n},\ep_{n},E^{n-1},E^{n})<\ve.
\end{equation}
Let us show that this choice is possible. Indeed, the strictly
convexity of $C$ guarantees that
\begin{equation}\label{equ: kappa}
\kappa_1(\partial E^{n-1})>\kappa_1(\partial C),\quad \forall
n\geq 2.
\end{equation}
Then, taking into account equation \eqref{eq:anita} and
inequalities \eqref{equ: tesis-7.7}, \eqref{equ: tesis-7.8} and
\eqref{equ: kappa} we conclude that
\[
{\tt m}(b_n,b_{n+1},\ep_{n+1},E^n,E^{n+1})<
 \frac{c_1}{n^2}
\left(1+2\sqrt{\frac{c_1^2}{n^4}+\frac2{\kappa_1(\partial C
)}}\right) < c_1\cdot \frac{\mathrm{c}}{n^2},\quad \forall n\in\N,
\]
where ${\rm c}$ is a constant which depends only on $C$.
Therefore, it can be assumed that $c_1$ was chosen small enough so
that \eqref{equ: suma m} holds.

Finally we define the desired immersion $\phi_\mu$. Taking into
account \eqref{equ: cierre}, \eqref{equ: suma m} and properties
(G$_n$), $n\in\N$, we infer that $\{X_n\}_{n\in\N}$ is a Cauchy
sequence uniformly in $\overline{M_\mu}$ of continuous maps.
Hence, it converges to a continuous map
$\Phi_\mu:\overline{M_\mu}\to \R^3$. Define
$\phi_\mu:=(\Phi_\mu)_{|M_\mu}:M_\mu\to\R^3$.

Let us check that $\phi_\mu$ satisfies the conclusion of the
theorem.

\noindent $\bullet$ Properties (D$_n$), $n\in\N$, guarantee that
$\phi_\mu$ is a conformal minimal immersion.

\noindent $\bullet$ The completeness of $\phi_\mu$ follows from
properties (B$_n$), (D$_n$), $n\in\N$, and the fact that the
sequence $\{1/\ep_n\}_{n\in\N}$ diverges.

\noindent $\bullet$ The properness of $\phi_\mu$ in $C$ is
equivalent to the fact that $\Phi_\mu(\partial M_\mu)\subset
\partial C$. Let us check it. Consider $p\in \partial M_\mu$. For
any $n\in\N$, let $p_n$ be a point in $M(\J_n^{\xi_n})$ such that
the sequence $\{p_n\}_{n\in\N}$ converges to $p$. This sequence of
points trivially exists since the definition of $M_\mu$. Fix
$k\in\N$. Then, the convex hull property for minimal surfaces and
(E$_n$) imply that $X_n(p_k)\in E^n$, for any $n\geq k$. Taking
limits as $n\to\infty$ we obtain that $\Phi_\mu (p_k)\in
\overline{C}$. Now, taking limits as $k\to\infty$, we conclude
that $\Phi_\mu(p)\in\overline{C}$. On the other hand,
$p\in\partial M_\mu\subset \overline{M(\J_n)}\setminus
M(\J_{n-1}^{\ep_n})$, $\forall n\in\N$. Again, fix $k\in\N$.
Properties (F$_n$), $n\in\N$, imply that $X_n(p)\in\R^3\setminus
(E^{k-1})_{-b_{k-1}-2b_k}$, for any $n>k$. Taking limits as
$n\to\infty$ we have that $\Phi_\mu(p)\in \overline{C}\setminus
(E^{k-1})_{-b_{k-1}-2b_k}$. Therefore, $\Phi_\mu(p)\in
\overline{C}\setminus (\cup_{k\in\N}
(E^{k-1})_{-b_{k-1}-2b_k})=\overline{C}\setminus C=\partial C$.


\noindent $\bullet$ Statement (i) follows from (A$_n$), $n\in\N$.

\noindent $\bullet$ Statement (ii) trivially holds.

\noindent $\bullet$ Taking into account \eqref{equ: cierre},
\eqref{equ: suma m} and properties (G$_n$), $n\in\N$, we conclude
that \begin{equation}\label{equ: <ve} \|\phi-\Phi_\mu\|<\ve\quad
\text{in $\overline{M_\mu}$}.
\end{equation}
This inequality implies statement (iii).

\noindent $\bullet$ Inequality \eqref{equ: ve<mu 2} implies that
$\de^H( \phi(M(\J^\ve))\,,\,\phi(\overline{M(\J)}) )<{\mu}/2.$
Then, to prove statement (iv) we use \eqref{equ: <ve}, the fact
that $M(\J^\ve)\subset\overline{M_\mu}\subset \overline{M(\J)}$
and the above inequality in the following way:
\begin{multline*}
\de^H\big( \Phi_\mu(\overline{M_\mu}) \;,\; \phi(\overline{M(\J)}) \big)< \de^H\big( \Phi_\mu(\overline{M_\mu}) \;,\; \phi(\overline{M_\mu}) \big) +
\\
 \de^H\big( \phi(\overline{M_\mu}) \;,\; \phi(\overline{M(\J)}) \big)<
\ve+ \de^H\big( \phi(M(\J^\ve)) \;,\; \phi(\overline{M(\J)}) \big)<\ve+\frac{\mu}2<\mu.
\end{multline*}

\noindent $\bullet$ Finally, let us check statement (v). Consider
$p\in \partial M_\mu$. Let $i\in\{1,\ldots,\ee\}$ such that $p\in
\mathcal{T}(\g_i)$ and label $q={\tt P}_i(p)\in\J$. Then
\[
\|\Phi_\mu(p)-\phi(q)\| < \|\Phi_\mu(p)-\phi(p)\|+ \|\phi(p)-\phi(q)\|<\ve+\frac{\mu}2<\mu,
\]
where we have used \eqref{equ: <ve} and \eqref{equ: ve<mu 2}. On
the other hand, given $q\in \J$ we can find a point $p\in
\partial M_\mu$ such that $q={\tt P}_i(p)$ for some
$i\in\{1,\ldots,\ee\}.$ The above computation gives
$\|\Phi_\mu(p)-\phi(q)\|<\mu$. In this way we have proved (v).

The proof of Theorem \ref{th: main} is done.


\subsection{Some consequences of Theorem \ref{th: main}}\label{subsec: coro}

In this subsection we remark some results that follow straightforwardly from Theorem \ref{th: main}.

The first one is a density type theorem. Given $C$ a strictly
convex bounded regular domain of $\R^3$, the set of finite
families of curves in $\partial C$ spanned by complete (connected)
minimal surfaces is dense in the set of finite families of curves
in $\partial C$ spanned by (connected) minimal surfaces, with the
Hausdorff metric.

\begin{corollary}\label{cor: Jordan}
Let $\Sigma$ be a finite family of closed curves in $\partial C$ so that
the Plateau problem for $\Sigma$ admits a solution. Then, for any
$\xi>0,$ there exist a compact Riemann surface $\mathcal{M},$ an open domain $M\subset \mathcal{M}$ and a continuous
map $\Phi:\overline{M}\to\overline{C}$ such that
\begin{enumerate}[1]
\item[$\bullet$] $\Phi_{|M}:M\to C$ is a conformal complete proper minimal immersion.

\item[$\bullet$] $\de^H(\Sigma,\Phi(\partial M))<\xi.$
\end{enumerate}
\end{corollary}

The following result shows that compact complete proper minimal
immersions in $C$ are not rare. Recall that any Riemann surface
with finite topology and analytic boundary can be seen as the
closure of an open region of a compact Riemann surface \cite{AS}.

\begin{corollary}\label{cor: densidad}
The family of complete minimal surfaces spanning a set of closed
curves in $\partial C$ is dense in the space of minimal surfaces
spanning a finite set of closed curves in $\partial C$, endowed
with the topology of the Hausdorff distance.
\end{corollary}

Theorem \ref{th: main} can be seen as an improvement of the
following result \cite[Theorem 3]{AFM}.

\begin{corollary}\label{cor: gafa-teo-3}
Let $C$ be a strictly convex bounded regular domain of $\R^3$.
Consider $\J$ a multicycle on the Riemann surface $M'$ and
$\vp:\overline{M(\J)}\to \overline{C}$ a conformal minimal
immersion satisfying $\vp(\J)\subset\partial C$.

Then, for any $\ep>0$, there exists a subdomain $M_\ep$ with the
same topological type as $M(\J)$, $\overline{M(\J^\ep)}\subset
M_\ep\subset \overline{M_\ep}\subset M(\J)$, and a complete proper
conformal minimal immersion $\vp_\ep:M_\ep\to C$ so that
\[
\|\vp-\vp_\ep\|<\ep,\quad \text{in $M_\ep$.}
\]
\end{corollary}


\subsection{The case of the disk}\label{subsec: disk}

As we remark in the introduction of this paper, any Jordan curve
in $\R^3$ can be spanned by a minimal disk. This fact allows us to
improve the preceding results in the simply-connected case.

\begin{corollary}\label{th: MM-2}
Let $C$ be a strictly convex bounded regular domain of $\R^3$. For any smooth Jordan curve $\g\subset \partial C$ and for any $\ep>0$ there exists a compact complete proper minimal immersion $\phi_{(\g,\ep)}:\D\to C$ such that $\de^H(\phi_{(\g,\ep)}(\partial\D),\g)<\ep$.
\end{corollary}

Now, we can prove the following existence result.

\begin{corollary}\label{cor: MM-3}
Let $D$ be a domain of $\R^3$ with boundary. Assume $\partial D$
contains a regular open connected region $A$ such that the mean
and Gauss curvatures are positive on $A$. Then, there exists a
compact complete proper minimal immersion $\phi:\D\to D$.
\end{corollary}

\begin{proof}
Consider $C$ a strictly convex bounded regular domain $C\subset D$ and such that $(\partial C)\cap (\partial D)$ is a non empty open subset of $A$ (see Figure \ref{fig: MM-4}).
 \begin{figure}[ht]
    \begin{center}
    \scalebox{0.4}{\includegraphics{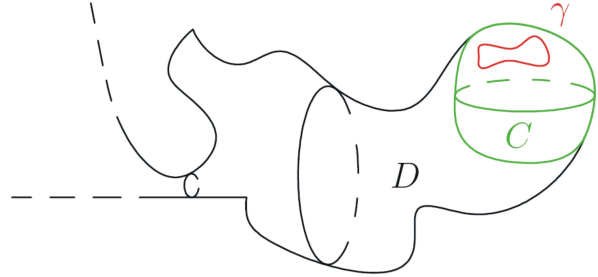}}
        \end{center}
\caption{The domains $D$ and $C$ and the curve $\g$.}\label{fig: MM-4}
\end{figure}
Then, apply Theorem \ref{th: MM-2} to a Jordan curve $\g\subset (\partial C)\cap (\partial D)$ and an $\ep>0$ small enough. Hence, we obtain a compact complete proper minimal immersion $\phi_{(\g,\ep)}:\D\to C\subset D$ whose boundary lies in $A\subset \partial D$. This is the immersion we are looking for.
\end{proof}

An interesting consequence follows from Corollary \ref{th: MM-2}
and the following well-known result in convex geometry \cite[Lemma
4]{MM2}.

\begin{lemma}\label{lem: MM-4}
Let $K$ be a connected compact set in $\partial C$, then for every $\nu>0$ there exists an smooth Jordan curve $\g\subset\partial C$ such that $\de^H(K,\g)<\nu$.
\end{lemma}

Next result essentially asserts that the Jordan curve $\g$ can be
substituted by an arbitrary compact set in the statement of
Corollary \ref{th: MM-2}.

\begin{corollary}\label{cor: MM-4}
Let $C$ be a strictly convex bounded regular domain, and consider a connected compact set $K\subset \partial C$. Then, for any $\ep>0$ there exists a compact complete proper minimal immersion $\phi_{(K,\ep)}:\D\to C$ satisfying that the Hausdorff distance $\de^H(\phi_{(K,\ep)}(\partial \D),K)<\ep$.
\end{corollary}

\begin{remark}
If we do not take care of the compactness of the immersions, then
the results stated in this subsection are those proved by
Mart\'{i}n and Morales in \cite{MM2}.
\end{remark}


\def\refname{References}

\vspace*{1cm}

\noindent
{\bf Antonio Alarc\'{o}n}\\
Departamento de Matem\'{a}tica Aplicada \\
Universidad de Murcia\\
E-30100 Espinardo, Murcia, Spain \\
e-mail: ant.alarcon@um.es



\end{document}